\documentclass{amsart}

\usepackage{mathrsfs} 

\usepackage{amsthm,amsmath,amssymb,amsfonts}

\usepackage{amscd}
\usepackage{amsmath}
\usepackage{mathtools}
\usepackage{bm}
\usepackage{amssymb}
\usepackage{graphicx}
\usepackage{psfrag}
\usepackage{verbatim}
\usepackage{tikz}
\usetikzlibrary{quotes}
\usetikzlibrary{topaths}
\usetikzlibrary{shapes,arrows,cd}
\usepackage{adjustbox}

\usepackage{enumerate}

\newtheorem{thm}{Theorem}[section]

\newtheorem{prop}[thm]{Proposition}

\theoremstyle{definition}

\theoremstyle{remark}

\theoremstyle{plain}

\numberwithin{equation}{section}

\theoremstyle{plain}

\begin{document}


\title{Four-dimensional Zero-Hopf Bifurcation for a Lorenz-Haken system }
\author[S. Renteria and P. Suarez]
{Sonia Renteria $^{1}$ and Pedro Suárez $^2$}

\address{$^1$ IME-USP Sao Paulo, Rua do Matao 1010, Cidade Universitaria, Sao Paulo, Brazil}
\email{srentalv@ime.usp.br}

\address{$^2$ Av. Professor Mello Moraes, 1235 - Butantã, São Paulo, Brazil}
\email{psuar86@gmail.com}

%




\begin{abstract}
In this work we study the periodic orbits which bifurcate from a zero-Hopf bifurcations that a Lorenz-Haken system in $\mathbb{R}^4$  can exhibit. The main tool used is the averaging theory.
\end{abstract}



\maketitle
\vspace{-0.5cm}


\section{Introduction and statement of the main results}

The Lorenz–Haken
equation named after the fluid dynamist Lorenz and
laser theorist Haken \cite{haken1975analogy}  describe the dynamics of a homogeneously broadened gain medium
in an unidirectional ring cavity. In the notation given in the Reference \cite{van1997nonlinear}, the Lorenz-Haken
equations is given by

\begin{equation}\label{ss0}
  \begin{aligned}
\dot{x} &= -\sigma ( x - y ) + i q x \vert x \vert ^2, \\
\dot{y} &= -(1 - i \delta ) y + (r - z) x , \\
\dot{z} &= -b z  + \mathcal{\mbox{Re}}(xy), \\
  \end{aligned}
\end{equation}
where $x,y$ and $z$ are complex variables, and $\sigma, b, q, r, \delta$ are the real parameters. In 2019, Hayder Natiq \cite{natiq2019dynamics} derived a new 4D chaotic laser system with three equilibrium points from (\ref{ss0}), 
since both $x$ and $z$ can be chosen to be real and $y$ a complex variable.

%

\smallskip
In this paper, we study a four-dimensional system of differential equations which is a generalization of the system introduced in \cite{natiq2019dynamics}.  We want to study the periodic orbits of the Lorenz-Haken systems of $\mathbb{R}^4$  with five parameters, in  which bifurcate in the zero-Hopf bifurcations of the singular points given by

\begin{equation}\label{s1}
  \begin{aligned}
\dot{x} &= a ( y - x ), \\
\dot{y} &= - c y - d z + (e - w) x , \\
\dot{z} &= d y - c z , \\
\dot{w} &= - b w + x y  ,
  \end{aligned}
\end{equation}
where $x, y, z, w$ are state variables and $a$, $b$, $c$, $d$ and $e$ are real parameters.

\smallskip

%


In the first instance we are going to compute the equilibrium points of Lorenz–Haken system (\ref{s1}). 

\begin{prop} \label{prr1}
Let  $\Delta= \dfrac{\big(  e c -c^{2} - d^{2}  \big)}{c} $ and  $c \neq 0$. The following statements are true:
\begin{enumerate}
\item If $\Delta\leq 0$ and  $b \neq 0$, system (\ref{s1}) has an unique equilibrium point $\mathtt{p_{0}}=(0,0,0,0)$.
\item If $\Delta>0$  and  $b \neq 0$, we have two equilibrium points
\begin{equation*}
  \begin{aligned}
\mathtt{p_{\pm}} &= \Bigg( \pm \sqrt{b \Delta}, \pm \sqrt{b \Delta}, \pm \dfrac{\sqrt{b \Delta}}{c} , \Delta \Bigg).
 \end{aligned}
\end{equation*}
\item If $b = 0$ and $\Delta \neq  0$ we has a straight line of equilibria
\begin{equation*}
  \begin{aligned}
\mathtt{p} &= \Bigg( 0, 0, 0, \Delta \Bigg). 
  \end{aligned}
\end{equation*}

\end{enumerate}
\end{prop}

Proposition \ref{prr1} follows easily by direct computations.

\medskip
%

We observe that the two equilibria $\mathtt{p_{\pm}}$ tends to the equilibrium point $\mathtt{p}$ when $b \rightarrow 0$. In short, the equilibrium point of system (\ref{s1}) can be $\mathtt{p_{+}}$, $\mathtt{p_{-}}$, $\mathtt{p}$ and the origin. Additionally, the system (\ref{s1}) has invariance under the coordinate transformation  $(x, y, z, w) \rightarrow (- x, - y, - z, w)$.   Consequently, the system (\ref{s1}) has rotational symmetry around the $w$-axis.

Due to that, in what follows we consider the only equilibrium $\mathtt{p_{+}}$ in order to verify its possibility of being a zero–Hopf equilibrium for some values of the parameter, and clearly the same will occur for the other equilibrium $\mathtt{p_{-}}$.

In the next result we characterize when the equilibrium $\mathtt{p}$, $\mathtt{p_{\pm}}$ and the origin are zero–Hopf equilibrium of the system (\ref{s1}). 


\begin{prop} \label{prop}
For the hyperchaotic system (\ref{s1}), the following statements hold:
\begin{enumerate}[(i)]
\item  $\mathtt{p_0}$ is a zero-Hopf equilibrium if only  if   $a=-2c,  b=0, d = - \frac{\sqrt{c^{2}+\omega^{2}}}{3} $ and 
$ \quad e = \frac{4 c^{2} + \omega^{2}}{3c}$, 
\item   $\mathtt{p}$ is  a zero-Hopf equilibrium if only if $a=-2c, b=0$ and $3d^2-c^2>0$,
\item $\mathtt{p_{+}}$ and $\mathtt{p_{-}}$ are zero-Hopf equilibrium if only if $a=-2c,  b=0, d= - \dfrac{\sqrt{c^{2}+\omega^{2}}}{\sqrt{3}}$.
\end{enumerate}

\end{prop}

In the rest of this section, we will study the zero-Hopf bifurcation and periodic solutions of the hyperchaotic system (\ref{s1}) at  all the equilibrium points.

\begin{thm}\label{teor2}
For the hyperchaotic system (\ref{s1}). The following statements hold.
\begin{enumerate}[(i)]
\item Let
$$(a, b, d, e) = \Bigg( - 2 c + \varepsilon a_{1} , \varepsilon b_{1}, - \frac{\sqrt{c^{2}+\omega^{2}}}{3} + \varepsilon d_{1}, \frac{4 c^{2} + \omega^{2}}{3c} + \varepsilon e_{1} \Bigg)$$

where $\omega > 0$ and $\varepsilon > 0$ are sufficiently small parameters.  If $a_{1}\neq 0$, $b_{1}\neq 0$, $c\neq 0$, $ \eta=3 c e_1+ 2 \sqrt{3} d_1 \sqrt{c^2+\omega^2} \neq 0$ and $\eta_1=3a_1 \omega^2- 2 c \eta \neq 0$,
then for $\varepsilon  > 0$ sufficiently small, the hyperchaotic system (\ref{s1}) has a zero-Hopf bifurcation at the equilibrium
point located at $\mathtt{p_{0}}$, and at most four periodic orbits can bifurcate from this equilibrium when $\varepsilon =0$. Moreover, the periodic solutions are stable if $a_1>0, b_1>0$, $16 \eta+3 b_1 \omega^2<0 $ and $ 4 \eta_1+3 b_1 \omega^2 <0$.
\item Let 
$$(a,b) = (- 2 c + \varepsilon a_{1} , \varepsilon b_{1}),$$

where $\omega > 0$ and $\varepsilon > 0$ are sufficiently small parameter. If $a_{1}\neq 0$, $b_{1}\neq 0$, $c\neq 0$, $d \neq 0$, $ (c^2 - d^2 )(c^2 + d^2 - c e) \neq 0$, $2(c^2 - d^2 ) -ce \neq 0$,  $ 3 d^2-c^2>0$, $c^{4} - 8 c^{2} d^{2} + 7 d^{4} + 2 c d^{2} e <0$ and $ (c^{4} - 4 c^{2} d^{2} + 3 d^{4})(c^{2} + d^{2} - c e )<0$, then for $\epsilon > 0$ sufficiently small, the hyperchaotic system (\ref{s1}) has a zero-Hopf bifurcation at the equilibrium
point located at $\mathtt{p} $, and at most five periodic orbits can bifurcate from this equilibrium when $\varepsilon=0$. Moreover, the periodic solution  are stable if $a_1>0, b_1>0$, $(c^{4}-8c^{2}d^{2}+7d^{4}+2cd^{2}e)<0$, $2c^{2}-2d^{2}-ce<0$ and $c^4-d^4-c^3 e+ cd^2 e >0$.
\item Let 

$$(a,b,d) = (- 2 c + \varepsilon a_{1} , \varepsilon b_{1}, - \dfrac{\sqrt{c^{2}+\omega^{2}}}{\sqrt{3}} + \varepsilon d_{1}),$$

where $\omega > 0$ and $\varepsilon > 0$ are sufficiently small parameter. If $c\neq 0$, $a_{1} \neq 0$, and $\kappa=b_{1}(4 c^2-3 ce+3 \omega^2)< 0$, then for $\varepsilon  > 0$ sufficiently small, the hyperchaotic system (\ref{s1}) has a zero-Hopf bifurcation at the equilibrium
point located at $\mathtt{p_{\pm}}$, and at most two periodic orbits can bifurcate from this equilibrium when $\varepsilon =0$. Moreover, the periodic solutions are unstable.
\end{enumerate}
\end{thm}

\section{The Averaging Theory of First Order}
The averaging theory is a classical and mature tool for studying the dynamic behavior of nonlinear dynamical systems, especially for the study of periodic solutions. This will be the main tool for proving Theorem \ref{teor2}.
\medskip

Consider differential system:
\begin{equation}\label{aver-1}
  \begin{aligned}
\dot{\mathbf{x}} &=  \varepsilon F(t, \mathbf{x}) + \varepsilon^{2} G(t, \mathbf{x}, \varepsilon),
  \end{aligned}
\end{equation}
where $\mathbf{x} \in D$ is an open subset of $\mathbb{R}^{n}$, $t \geq 0$. We assume that $F(t, \mathbf{x})$ and $G(t, \mathbf{x}, \varepsilon)$ are $T$-periodic in $t$. We define averaged function
\begin{equation}\label{aver-2}
  \begin{aligned}
f(\mathbf{x}) &= \frac{1}{T} \int^{T}_{0} F(t, \mathbf{x}) d t.
  \end{aligned}
\end{equation}
\begin{thm}\label{teor3}
Make the following assumptions:
\begin{enumerate}[(i)]
\item $F$, its Jacobian $\dfrac{\partial F}{\partial \mathbf{x}}$ and its Hessian $\dfrac{\partial^{2} F}{\partial \mathbf{x}^{2}}$; $G$, its Jacobian $\dfrac{\partial G}{\partial \mathbf{x}}$ are defined, continuous and bounded by a constant independent of $\varepsilon$ in $[0,\infty) \times D$ and $\varepsilon \in (0,\varepsilon_{0}]$.
\item $T$ is a constant independent of $\varepsilon$.
\end{enumerate}
Then the following conclusions can be obtained:
\begin{enumerate}[(a)]
\item If $p$ is the zero of the averaged function $f(\mathbf{x})$, and
\begin{equation}\label{aver-3}
  \begin{aligned}
\det \Big( \dfrac{\partial f}{\partial \mathbf{x}} \Big) \Big|_{\mathbf{x} = p} \neq 0,
  \end{aligned}
\end{equation}
then there exists a $T$-periodic solution $\mathbf{x}(t,\varepsilon)$ of system (\ref{aver-1}) such that $\mathbf{x}(0,\varepsilon) \rightarrow p$ as $\varepsilon \rightarrow 0$.
\item If the eigenvalue of the Jacobian matrix $\Big( \dfrac{\partial f}{\partial \mathbf{x}} \Big)$ has a negative real part, the periodic solution $\mathbf{x}(t,\varepsilon)$ is asymptotically stable.
\end{enumerate}
\end{thm}

For more information about the averaging theory see  \cite{sanders2007averaging} and \cite{verhulst2006nonlinear}.
\section{Proof of results}
In this section we will provide the proofs of Proposition \ref{prop} and Theorem \ref{teor2}.
\begin{proof}[Proof of Proposition \ref{prop}]
The characteristic polynomial $P(\lambda)$ of the linear part of the differential systems  (\ref{s1}) at the equilibrium point $p_{0} = (0, 0, 0, 0)$ is 
\begin{equation}\label{eq-proof-1}
  \begin{aligned}
  P(\lambda)=\lambda^{4} + A \lambda^{3} + B \lambda^{2} + C \lambda + D,
  \end{aligned}
\end{equation}
where 
\begin{align*}
A &=a+b+2c, \\
B &=2 b c + c^{2} + d^{2} + a(b + 2 c - e), \\
C &= b(c^{2} + d^{2}) + a \big( 2 b c + c^{2} + d^{2} - (b - c) e \big), \\
D &= a b \big( c^{2} + d^{2} - c e \big).
\end{align*}

The  equilibrium point $p_0$ is a zero hopf equilibrium if and only if  $P(\lambda)=\lambda^2(\lambda^2+ \omega^2)$ with $\omega >0$, the parameter must be satisfied,
 $$a=-2c, \quad b=0, \quad d = - \frac{\sqrt{c^{2}+\omega^{2}}}{3} \quad \mbox{ and } \quad e = \frac{4 c^{2} + \omega^{2}}{3c}.$$

(ii) The characteristic polynomial $P(\lambda)$ of the linear part of the differential systems  (\ref{s1}) at the equilibrium point $\mathtt{p}$ is 
\begin{equation}\label{poly2}
  \begin{aligned}
  P(\lambda)=\lambda^{4} + (a+ 2 c) \lambda^{3} + \Big(c^2+d^2+a(c-\dfrac{d^2}{c}) \Big) \lambda^{2}.
  \end{aligned}
\end{equation}

The  equilibrium point $\mathtt{p}$ is a zero hopf equilibrium if and only if  $P(\lambda)=\lambda^2(\lambda^2+ \omega^2)$ with $\omega >0$, the parameter must be satisfied,
 $$a=-2c, \quad b=0, $$
in this case, Eq. (\ref{poly2}) has roots $\lambda_{1,2}=0$, $\lambda_{3,4}= \pm \sqrt{3d^2-c^2}i$.
\medskip

(iii) The Jacobian matrix of systems (\ref{s1}) evaluated at $p_{+}$ is 

$$\left( \begin{array}{cccc}
-a & a & 0 & 0 \\
c+ \dfrac{d^2}{c}& -c & -d & -\dfrac{\sqrt{b(ce-c^2-d^2)}}{\sqrt{c}} \\
0 & d & -c & 0 \\
\dfrac{\sqrt{b(ce-c^2-d^2)}}{\sqrt{c}} & \dfrac{\sqrt{b(ce-c^2-d^2)}}{\sqrt{c}} & 0 & -b
\end{array}\right)$$
ant its characteristic polynomial is 
\begin{equation}\label{eq-proof-1}
  \begin{aligned}
  P(\lambda)=\lambda^{4} + A \lambda^{3} + B \lambda^{2} + C \lambda + D,
  \end{aligned}
\end{equation}
where 
\begin{align*}
A &=a+b+2c, \\
B &= c^{2} + d^{2} + a(b+c-\dfrac{d^2}{c})+b(c-\frac{d^2}{c}+e), \\
C &= b\Big( ce + a (-c-\dfrac{3d^2}{c}+2e) \Big), \\
D &= -2 a b \big( c^{2} + d^{2} - c e \big).
\end{align*}

The  equilibrium point $\mathtt{p_{+}}$ is a zero hopf equilibrium if and only if  $P(\lambda)=\lambda^2(\lambda^2+ \omega^2)$ with $\omega >0$, the parameter must be satisfied,
 $$a=-2c, \quad b=0, \quad d = - \frac{\sqrt{c^{2}+\omega^{2}}}{\sqrt{3}}.$$
This completes the Proof of Proposition \ref{prop}.
\end{proof}


\begin{proof}[Proof. of statement (i) of Theorem \ref{teor2}]
Let
$$(a, b, d, e) = \Bigg( - 2 c + \varepsilon a_{1} , \varepsilon b_{1}, - \frac{\sqrt{c^{2}+\omega^{2}}}{3} + \varepsilon d_{1}, \frac{4 c^{2} + \omega^{2}}{3c} + \varepsilon e_{1} \Bigg)$$

where $\omega > 0$ and $\varepsilon > 0$ are sufficiently small parameters. Then, the differential systems  (\ref{s1}) becomes

\begin{equation}\label{eq-proof-2}
  \begin{aligned}
\dot{x} &= 2 c (x - y) - a_{1}(x - y) \varepsilon , \\
\dot{y} &= (e_{1} x - d_{1} z) \varepsilon - \frac{-4c^{2}x+3cwx+3c^{2}y-x\omega^{2}-\sqrt{3}cz\sqrt{c^{2}+\omega^{2}}}{3 c} , \\
\dot{z} &= d_{1} y \varepsilon + \frac{1}{3} \big( -3cz-\sqrt{3} y \sqrt{c^{2}+\omega^{2}} \big) , \\
\dot{w} &= x y - b_{1} w \varepsilon.
  \end{aligned}
\end{equation} 
Performing the rescaling of variables
\begin{equation*}
(x, y, z, w) \mapsto (\varepsilon x, \varepsilon y, \varepsilon z, \varepsilon w)
\end{equation*}
system (\ref{eq-proof-2}) can be written as
\begin{equation}\label{eq-proof-3}
 \begin{aligned}
\dot{x} &= 2 c(x-y) - a_{1}(x - y) \varepsilon , \\
\dot{y} &= (e_{1} x - w x - d_{1} z) \varepsilon - \frac{-4c^{2}x+3c^{2}y-x\omega^{2}-\sqrt{3}cz\sqrt{c^{2}+\omega^{2}}}{3 c}, \\
\dot{z} &= d_{1} y \varepsilon + \frac{1}{3} \Big( - 3 c z - \sqrt{3} y \sqrt{c^{2}+\omega^{2}}\Big), \\
\dot{w} &= (- b_{1} w + x y) \varepsilon.
  \end{aligned}
\end{equation} 
Now we shall write the linear part at the origin of the system (\ref{eq-proof-3}) when $\varepsilon = 0$ into its real Jordan normal form, i.e. as
$$\left( \begin{array}{cccc}
0 & - \omega & 0 & 0 \\
\omega & 0 & 0 & 0 \\
0 & 0 & 0 & 0 \\
0 & 0 & 0 & 0
\end{array}\right).$$
For doing that we consider the linear change $(x,y,z,w) \mapsto (X,Y,Z,W)$
\begin{equation}\label{eq-proof-4}
\begin{aligned}
x &= \frac{2c \Big( \sqrt{3}cY\omega+\sqrt{3}X\omega^{2}-3cZ\sqrt{c^{2}+\omega^{2}} \Big)}{3\omega^{2}\sqrt{c^{2}+\omega^{2}}}, \nonumber \\
y &= \frac{\sqrt{3}cX\omega^{2}+\sqrt{3}Y\omega^{3}+2c^{2} \Big( \sqrt{3}Y\omega-3Z\sqrt{c^{2}+\omega^{2}} \Big)}{3\omega^{2}\sqrt{c^{2}+\omega^{2}}}, \\
z &= \frac{1}{3} \Bigg( X + \frac{c \Big( -2Y\omega+2\sqrt{3}Z\sqrt{c^{2}+\omega^{2}} \Big)}{\omega^{2}} \Bigg), \nonumber \\
w &= W. \nonumber
\end{aligned}
\end{equation}
By using the new variables $(X, Y, Z, W)$, the system (\ref{eq-proof-3}) can be written as follows:

\begin{equation}\label{eq-proof-5}
 \begin{aligned}
\dot{X} &= - Y \omega  + \frac{1}{3} \varepsilon \Bigg( a_{1} \big( - X + \frac{Y \omega}{c} \big) + \frac{d_{1}}{\omega^{2} \sqrt{c^{2}+\omega^{2}}}  \big( - 6 c^{2} \sqrt{c^{2}+\omega^{2}} Z + \sqrt{3} \omega (2c^{2}Y \\
&\quad + c X \omega + Y \omega^{2}) \big) \Bigg), \\
\dot{Y} &= X \omega + \frac{\varepsilon}{3 \omega^{3} \sqrt{c^{2}+\omega^{2}}} \Bigg( 6\sqrt{3}c^{4}(-e_{1}+W)Z - 6c^{2}(e_{1}-W)\omega (\sqrt{3}Z\omega-Y\sqrt{c^{2}+\omega^{2}}) \\ 
&\quad - \omega^{3}(\sqrt{3}d_{1}X\omega+2a_{1}Y\sqrt{c^{2}+\omega^{2}}) + 4c^{3}d_{1}(\sqrt{3}Y\omega - 3Z\sqrt{c^{2}+\omega^{2}})  \\
 &\quad  + c\omega^{2}\Big( 2(a_{1}+3e_{1} -3W)X\sqrt{c^{2}+\omega^{2}} + 3d_{1} (\sqrt{3}Y \omega - 2Z\sqrt{c^{2}+\omega^{2}}) \Big) \Bigg), \\
\dot{Z} &= \frac{\varepsilon}{18c^{2}\omega^{2}\sqrt{c^{2}+\omega^{2}}}  \Bigg( -24\sqrt{3}c^{5}d_{1}Z - 4\sqrt{3}a_{1}c^{2}Y\omega^{3} - \sqrt{3}a_{1}Y\omega^{5} + c\omega^{3} (\sqrt{3}a_{1}X\omega \\
 &\quad +6d_{1}Y \sqrt{c^{2}+\omega^{2}}) + 4c^{3} \omega \Big( \sqrt{3} \big( (a_{1}+3e_{1}-3W)X-6d_{1}Z \big)\omega + ed_{1}Y \sqrt{c^{2}+\omega^{2}} \Big)  \\ 
 &\quad  + 12c^{4}(e_{1}-W) (\sqrt{3}Y \omega - 3Z\sqrt{c^{2}+\omega^{2}})\Bigg), \\
\dot{W} &= \varepsilon \Bigg( -b_{1} W + \frac{2c}{9\omega^{4}(c^{2}+\omega^{2})} \big(\sqrt{3}(cY\omega+X\omega^{2})-3cZ\sqrt{c^{2}+\omega^{2}} \big) \big( \sqrt{3}(cX\omega^{2} + Y\omega^{3})  \\
&+ 2c^{2} (\sqrt{3}Y\omega - 3Z\sqrt{c^{2}+\omega^{2}}) \big) \Bigg).
  \end{aligned}
\end{equation}

Then we use the cylindrical coordinates $X = r \cos \theta$, $Y = r \sin \theta$, and obtain 
\begin{equation}\label{eq-proof-6}
 \begin{aligned}
\dot{r} &= \frac{\varepsilon}{3c\omega^{3} \sqrt{c^{2}+\omega^{2}}} \Bigg( c r \omega^{3}(\sqrt{3} c d_{1}-a_{1}\sqrt{c^{2}+\omega^{2}})\cos^{2}\theta + c \sin \theta \Big( 6 c Z \big( -(2 c^{2}  \\
 &\quad + \omega^{2}) d_{1}  \sqrt{c^{2}+\omega^{2}} - \sqrt{3}c(e_{1}-W)(c^{2}+\omega^{2}) \big) + r \omega \big( \sqrt{3}(4c^{3} + 3 c \omega^{2}) d_{1} \\
  &\quad  + (6c^{2}(e_{1}-W) - 2a_{1}\omega^{2})  \sqrt{c^{2}+\omega^{2}} \big) \sin \theta \Big) + \omega \cos \theta \Big( -6 c^{3} d_{1} Z \sqrt{c^{2}+\omega^{2}} \\
   &\quad + r \omega \big( 2 \sqrt{3} c^{3} d_{1} + 2 c^{2} (a_{1}+3e_{1}-3W) \sqrt{c^{2}+\omega^{2}} + a_{1} \omega^{2} \sqrt{c^{2}+\omega^{2}} \big) \sin \theta \Big) \Bigg), \\
\dot{\theta} &= \frac{1}{3 c r \omega^{3} \sqrt{c^{2}+\omega^{2}}} \Bigg( cr \omega^{2}\sqrt{c^{2}+\omega^{2}} \big( 2 c (a_{1}+3e_{1}-3W)\varepsilon+3\omega^{2} \big) \cos^{2} \theta  \\ 
&\quad  + c \varepsilon \cos \theta \Big( 6 c Z  \big( -(2 c^{2}+\omega^{2})d_{1}\sqrt{c^{2}+\omega^{2}} - \sqrt{3} c (e_{1}-W)(c^{2}+\omega^{2}) \big)  \\
 &\quad + r \omega \big( 4 \sqrt{3} c^{3} d_{1} + 6c^{2}(e_{1}-W)   \sqrt{c^{2}+\omega^{2}} - a_{1} \omega^{2} \sqrt{c^{2}+\omega^{2}} \big) \sin \theta \Big)  \\
  &\quad + \omega \Big( 6 c^{3} d_{1} Z \varepsilon \sqrt{c^{2}+\omega^{2}} \sin \theta + r \omega \big( - 2 \sqrt{3} c^{3} d_{1} \varepsilon +  (3 c - a_{1} \varepsilon) \omega^{2} \sqrt{c^{2}+\omega^{2}} \big) \sin^{2} \theta \\
  & \quad + \sqrt{3} c d_{1} r \varepsilon \omega^{2} (- \omega + c \sin 2 \theta) \Big) \Bigg), \\
\dot{Z} &= \frac{\varepsilon}{18 c^{2} \omega^{2} \sqrt{c^{2}+\omega^{2}}}  \Bigg( - 12 c^{3} Z \big( 2 \sqrt{3}(c^{2}+\omega^{2})d_{1} + 3 c (e_{1}-W) \sqrt{c^{2}+\omega^{2}} \big)  \\
 &\quad + \sqrt{3} c r \omega^{2}  \Big( 4 c^{2} (a_{1}+3e_{1}-3W) + a_{1} \omega^{2} \Big) \cos \theta + r \omega \Big( (24 c^{3} + 6 c \omega^{2})d_{1} \sqrt{c^{2}+\omega^{2}}  \\ 
  &\quad  - \sqrt{3} \big( 12 c^{4} (- e_{1} + W) + 4 a_{1} c^{2} \omega^{2} + a_{1} \omega^{4} \big) \Big)  \sin \theta \Bigg), \\
\dot{W} &= \frac{\varepsilon}{ 3 \omega^{4} (c^{2}+\omega^{2})} \Bigg( \big( 12 c^{4} Z^{2} + 2 c^{2} r^{2} \omega^{2} - 3 b_{1} W \omega^{4} \big) (c^{2}+\omega^{2}) + c r \omega \Big( - 2 c^{3} r \omega \cos 2 \theta  \\ 
&\quad  - 2 \sqrt{3} c Z \sqrt{c^{2}+\omega^{2}} \big( 3 c \omega \cos \theta + (4 c^{2}+\omega^{2}) \sin \theta \big) + 3 c^{2} r \omega^{2} \sin 2 \theta + r \omega^{4} \sin 2 \theta \Big) \Bigg).
  \end{aligned}
\end{equation}
We take $\theta$ as a new independent variable and obtain the system 
\begin{equation}\label{eq-proof-7}
 \begin{aligned} 
\frac{d r}{d \theta} &=  \frac{\varepsilon}{3 c \omega^{4} \sqrt{c^{2}+\omega^{2}}} \Bigg( c r \omega^{3} (\sqrt{3} c d_{1} - a_{1} \sqrt{c^{2}+\omega^{2}}) \cos^{2} \theta + \omega \sqrt{c^{2}+\omega^{2}} \cos \theta \\
 &\quad	\Big( - 6 c^{3} d_{1} Z 	+ r \omega \big(6 c^{2} (e_{1}-W) + a_{1} \omega^{2} \big) \sin \theta \Big) + c \sin \theta \bigg( 6 c z \big( - (2 c^{2} 	\\ 
 &\quad	+ \omega^{2}) d_{1} \sqrt{c^{2}+\omega^{2}} - \sqrt{3} c (e_{1} - W) (c^{2}+\omega^{2}) \big) + r \omega \Big( 2 c \omega (\sqrt{3} c d_{1}  \\
  &\quad  + a_{1}\sqrt{c^{2}+\omega^{2}}) \cos \theta + \big( (4 c^{3} + 3 c \omega^{2}) \sqrt{3} d_{1} + \big( 6 c^{2} (e_{1}-W)  - 2 a_{1} \omega^{2} \big) \sqrt{c^{2}+\omega^{2}} \Big) \\
 &\quad \sin \theta \bigg) \Bigg)  + O (\varepsilon^{2}) \\
	&=  \varepsilon F_{1} (\theta, r, Z, W) + O (\varepsilon^{2}), \\	
\frac{d Z}{d \theta} &=  \frac{\varepsilon}{18 c^{2} \omega^{3} \sqrt{c^{2}+\omega^{2}}} \Bigg( - 12 c^{3} Z \Big( 2 \sqrt{3} (c^{2} + \omega^{2}) d_{1} + 3 c (e_{1}-W) \sqrt{c^{2}+\omega^{2}} \Big) 	\\
 &\quad + \sqrt{3} c r \omega^{2} 	\Big( 4 c^{2} (a_{1} + 3 e_{1} - 3 W) + a_{1} \omega^{2} \Big) \cos \theta + r \omega \Big( 6 (4 c^{3} + c \omega^{2}) d_{1} \sqrt{c^{2}+\omega^{2}} 	\\
  &\quad - \sqrt{3} \big( 12 c^{4} (- e_{1} + W) 	+ 4 a_{1} c^{2} \omega^{2} + a_{1} \omega^{4} \big) \Big) \sin \theta \Bigg) + O (\varepsilon^{2}) \\
	&=  \varepsilon F_{2} (\theta, r, Z, W) + O (\varepsilon^{2}), \\	
\frac{d W}{d \theta} &=  \frac{\varepsilon}{ 3 \omega^{5} (c^{2}+\omega^{2})} \Bigg( (c^{2}+\omega^{2}) \Big( 12 c^{4} Z^{2} + 2 c^{2} r^{2} \omega^{2} - 3 b_{1} W \omega^{4} \Big) + c r \omega \Big( - 2 c^{3} r \omega \cos 2 \theta \\ 
&\quad  - 2 \sqrt{3} c Z \sqrt{c^{2}+\omega^{2}} \big( 3 c \omega \cos \theta + (4 c^{2} + \omega^{2}) \sin \theta \big) + 3 c^{2} r \omega^{2} \sin 2 \theta \\
&\quad + r \omega^{4} \sin 2 \theta \Big) \Bigg) + O (\varepsilon^{2}) \\
	&=  \varepsilon F_{3} (\theta, r, Z, W) + O (\varepsilon^{2}).	
  \end{aligned}
\end{equation}
Using the notation of averaging theory introduced in Theorem \ref{teor3}, we get $t = \theta$, $T = 2 \pi$, $\mathbf{x} = (r, Z, W)$ and
$$
F(\theta, r, Z, W) = \left( \begin{array}{c}
F_{1}(\theta, r, Z, W) \\
F_{2}(\theta, r, Z, W) \\
F_{3}(\theta, r, Z, W)
\end{array}\right)
\mbox{,} \hspace{0.2cm} and \hspace{0.2cm}
f(r, Z, W) = \left( \begin{array}{c}
f_{1}(r, Z, W) \\
f_{2}(r, Z, W) \\
f_{3}(r, Z, W)
\end{array}\right).
$$
Then we compute the integrals, i.e.
\begin{equation*}
 \begin{aligned}
f_{1}(r, Z, W) &= \frac{1}{2 \pi} \int^{2 \pi}_{0} F_{1}(\theta, r, Z, W) d \theta = \frac{r \Big( 6 c^{2} (e_{1}-W) - 3 a_{1} \omega^{2} + 4 \sqrt{3} c d_{1} \sqrt{c^{2}+\omega^{2}} \Big)}{6 \omega^{3}} , \\ 
f_{2}(r, Z, W) &= \frac{1}{2 \pi} \int^{2 \pi}_{0} F_{2}(\theta, r, Z, W) d \theta = - \frac{2 c Z \Big(3 c (e_{1}-W) + 2 \sqrt{3} d_{1} \sqrt{c^{2}+\omega^{2}} \Big)}{3 \omega^{3}} , \\ 
f_{3}(r, Z, W) &= \frac{1}{2 \pi} \int^{2 \pi}_{0} F_{3}(\theta, r, Z, W) d \theta = \frac{12 c^{4} Z^{2} + 2 c^{2} r^{2} \omega^{2} - 3 b_{1} W \omega^{4}}{3 \omega^{5}}. \\ 
  \end{aligned}
\end{equation*}
Solving the equations $f_{1}(r, Z, W) = f_{2}(r, Z, W) = f_{3}(r, Z, W) = 0$, we can get the following five solutions:
\begin{equation*}
 \begin{aligned}
s_{0} &= (0, 0, 0), \\
s_{1,2} &= \Bigg(0, \mp \dfrac{\sqrt{b_{1} \omega^{4} \Big(3 c e_{1} + 2 \sqrt{3} d_{1} \sqrt{c^{2}+\omega^{2}} \Big)}}{2 \sqrt{3} c^{5/2}},  e_{1} + \dfrac{2 d_{1} \sqrt{c^{2}+\omega^{2}}}{\sqrt{3} c} \Bigg), \\
s_{3,4} &= \Bigg( \mp \dfrac{\sqrt{b_{1} \omega^{2} \Big( 6 c^{2} e_{1} - 3 a_{1} \omega^{2} + 4 \sqrt{3} c d_{1} \sqrt{c^{2}+\omega^{2}} \Big)}}{2 c^{2}}, 0 , \hspace{0.2cm} e_{1} + \dfrac{1}{6 c^{2}}(- 3 a_{1} \omega^{2} \\
&\quad + 4 \sqrt{3} c d_{1} \sqrt{c^{2}+\omega^{2}}) \Bigg). \\
  \end{aligned}
\end{equation*}

The first solution $s_0$ corresponds to the equilibrium at the origin. For other four solutions, we get

\begin{enumerate}[(I)]
\item For the solution $s_{1}$ and $s_{2}$ when $c \neq 0$, $s_{1,2}$ are real solutions. The Jacobian of solution $s_{1,2}$ is
\begin{equation*}
 \begin{aligned}
\det \Big( \frac{\partial f}{\partial \mathbf{x}} (s_{1}) \Big) &= \det \Big( \frac{\partial f}{\partial \mathbf{x}} (s_{2}) \Big) \\
&= \frac{2 a_{1} b_{1} c \big( 3 c e_{1} + 2 \sqrt{3} d_{1} \sqrt{c^{2}+\omega^{2}} \big)}{3 \omega^{5}} .
  \end{aligned}
\end{equation*}
\item For the solution $s_{3}$ and $s_{4}$ when $c \neq 0$, $s_{3,4}$ are real solutions. The Jacobian of solution $s_{3,4}$ is
\begin{equation*}
 \begin{aligned}
\det \Big( \frac{\partial f}{\partial \mathbf{x}} (s_{3}) \Big) &= \det \Big( \frac{\partial f}{\partial \mathbf{x}} (s_{4}) \Big) \\
&= \frac{a_{1} b_{1} \big( -6 c^{2} e_{1} + 3 a_{1} \omega^{2} - 4 \sqrt{3} c d_{1} \sqrt{c^{2}+\omega^{2}} \big)}{3 \omega^{5}} .
  \end{aligned}
\end{equation*}

When $a_{1} \neq 0$, $b_{1} \neq 0$, $c \neq 0$, $\eta= 3 c e_{1} + 2 \sqrt{3} d_{1} \sqrt{c^{2}+\omega^{2}} \neq 0$ and $\eta_1=3a_1 \omega^2-2 c \eta \neq 0$, then $\det \Big( \frac{\partial f}{\partial \mathbf{x}} (s_{j}) \Big) \neq 0$, $j=1,\ldots 4$. 
Then according to Theorem \ref{teor3}, we see that the system (\ref{eq-proof-7}) has one periodic solution $\mathbf{x}_{j}(\theta , \varepsilon)$ such that $\mathbf{x}_{j}(0 , \varepsilon) = s_{j} + O(\varepsilon)$, $j=1,\ldots 4$. Bring the solution back to the system (\ref{eq-proof-5}), and we have one periodic solution 
$\Phi_{j}(t,\varepsilon) = \big( X_{j}(t,\varepsilon), Y_{j}(t,\varepsilon), Z_{j}(t,\varepsilon), W_{j}(t,\varepsilon) \big)$.
Then the system (\ref{eq-proof-2}) has the periodic solution $\varepsilon \Phi_{j}(t,\varepsilon)$, $j=1,\ldots 4$.

\end{enumerate}

\medskip

To determine the stability of the periodic solution $\varepsilon \Phi_{j}(t,\varepsilon)$, $j=1, \ldots,4$, one needs to calculate eigenvalues of the Jacobian matrix $\frac{\partial f}{\partial \mathbf{x}}(s_{2,3})$.
%

\begin{equation}\label{estab1}
 \begin{aligned}
P(s_{2,3}) &= c_{0} \lambda^{3} + c_{1} \lambda^{2} + c_{2} \lambda + c_{3}
  \end{aligned}
\end{equation}
where $c_{0}$, $c_{1}$, $c_{2}$ and $c_{3}$ are 
\begin{equation*}
 \begin{aligned}
c_{0} &= - 1, \\ 
c_{1} &= - \frac{a_{1} + 2 b_{1}}{2 \omega}, \\
c_{2} &= \frac{b_{1} \Big(- 3 a_{1} \omega^{2} + 8 c \big( 3 c e_{1} + 2 \sqrt{3} d_{1} \sqrt{c^{2}+\omega^{2}} \big) \Big)}{6 \omega^{4}}, \\
c_{3} &= \frac{2 a_{1} b_{1} c \Big( 3 c e_{1} + 2 \sqrt{3} d_{1} \sqrt{c^{2}+\omega^{2}} \Big)}{3 \omega^{5}}.
  \end{aligned}
\end{equation*}

The eigenvalues are  given as follows:

\begin{equation*}
 \begin{aligned}
\lambda_{1} &= - \frac{a_{1}}{2 \omega},  \hspace{0,2cm} \lambda_{2,3} = - \frac{3 b_{1} \pm \sqrt{3} \omega \sqrt{\dfrac{b_{1} \Big(  48 c^{2} e_{1} + 3 b_{1} \omega^{2} + 32 \sqrt{3} c d_{1} \sqrt{c^{2} + \omega^{2}} \Big)}{\omega^{4}}}}{6 \omega^{3}}. 
  \end{aligned}
\end{equation*}

On the other hand, the characteristic polynomial and its eigenvalues of the Jacobian matrix $\frac{\partial f}{\partial \mathbf{x}}(s_{3,4})$ are

\begin{equation}\label{estab1}
 \begin{aligned}
P(s_{3,4}) &= c_{0} \lambda^{3} + c_{1} \lambda^{2} + c_{2} \lambda + c_{3}
  \end{aligned}
\end{equation}
where $c_{0}$, $c_{1}$, $c_{2}$ and $c_{3}$ are 
\begin{equation*}
 \begin{aligned}
c_{0} &= - 1, \\ 
c_{1} &= - \frac{a_{1} + b_{1}}{\omega}, \\
c_{2} &= - \frac{2 b_{1} c \big( 3 c e_{1} + 2 \sqrt{3} d_{1} \sqrt{c^{2}+\omega^{2}} \big)}{3 \omega^{4}}, \\
c_{3} &= \frac{a_{1} b_{1} \big( - 6 c^{2} e_{1} + 3 a_{1} \omega^{2} - 4 \sqrt{3} c d_{1} \sqrt{c^{2}+\omega^{2}} \big)}{3 \omega^{5}}.
  \end{aligned}
\end{equation*}

The eigenvalues are  given as follows:


\begin{equation*}
 \begin{aligned}
\widetilde{\lambda_{1}} &= - \frac{a_{1}}{\omega}, \hspace{0,2cm} \widetilde{\lambda_{2,3}} = - \frac{3 b_{1} \omega^{3} \pm \sqrt{3} \sqrt{b_{1} \omega^{4} \Big( 3 (4 a_{1} + b_{1}) \omega^{2} - 8 c \big(3 c e_{1} + 2 \sqrt{3} d_{1} \sqrt{c^{2} + \omega^{2}} \big) \Big)}}{6 \omega^{3}}. 
  \end{aligned}
\end{equation*}

We have that $\lambda_{1},\widetilde{\lambda_{1}} $ is real and $\lambda_{2,3}, \widetilde{\lambda_{2,3}}$ are complex numbers if $16 \eta +3 b_1 \omega^2<0$  and $4 \eta_{1}+3 b_1 \omega^2<0$. In this case, the periodic solution $\varepsilon \Phi_{j} (t,\varepsilon)$ is stable if $a_{1} > 0$, $b_{1} > 0$. 
\end{proof}


\begin{proof}[ Proof. of statement (ii) of Theorem \ref{teor2}]
Let 
$$(a,b) = (- 2 c + \varepsilon a_{1} , \varepsilon b_{1}),$$
where $\omega > 0$ and $\varepsilon > 0$ are sufficiently small parameter. Them, we translate $\mathtt{p}$ to the origin the coordinates doing 
system (\ref{s1}) becomes $(x,y,z,w)=(\overline{x},\overline{y},\overline{z},\overline{w})+ \mathtt{p}$, then we introduce the scaling of variables $(\overline{x},\overline{y},\overline{z},\overline{w}) = (\varepsilon x,\varepsilon y,\varepsilon z,\varepsilon w)$, with these changes of variables  system (\ref{s1}) can be written as 
\begin{equation}\label{p2}
 \begin{aligned}
\dot{x} &= 2 c (x-y) - a_{1} (x-y) \varepsilon, \\ 
\dot{y} &= \dfrac{c^{2} x + d^{2} x - c^{2} y - c d z}{c} - w x \varepsilon, \\ 
\dot{z} &= d y -  c z, \\
\dot{w} &= \dfrac{b_{1} (c^{2} + d^{2} - c e)}{c} + (- b_1 w + x y) \varepsilon. \\ 
 \end{aligned}
\end{equation}
After the linear change in variables $(x,y,z,w) \mapsto (X, Y, Z, W)$, 
\begin{equation}\label{p3}
 \begin{aligned}
x &= \dfrac{-6d^{2}X+2c^{2}(3W+X)-2c\sqrt{-c^{2}+3d^{2}}Y}{3(c^{2}+3d^{2})}, \\ 
y &= - \dfrac{1}{3(c^{3}-3cd^{2})} \Big( 3 c d^{2} X - c^{3} (6 W + X) + \sqrt{- c^{2} + 3 d^{2}} (c^{2} + 3 d^{2})  Y \Big), \\
z &= \dfrac{d \Big( - 3 d^{2} X + c^{2} (-6 W + X) + 2 c \sqrt{-c^{2} + 3 d^{2}} Y \Big)}{- 3 c^{3} + 9 c d^{2}}, \\ 
w &= Z. 
 \end{aligned}
\end{equation}
the linear part at the origin of system (\ref{p2}) for $\varepsilon = 0$ can be transformed into its real Jordan normal form,
$$\left( \begin{array}{cccc}
0 & - \sqrt{3d^{2}-c^{2}} & 0 & 0 \\
\sqrt{3d^{2}-c^{2}} & 0 & 0 & 0 \\
0 & 0 & 0 & 0 \\
0 & 0 & 0 & 0
\end{array}\right).$$
Under the change in variable (\ref{p3}), where we have written $(x,y,z,w)$ instead of $(X,Y,Z,W)$ the system (\ref{p2}) can be written as
\begin{equation}\label{p4}
 \begin{aligned}
\dot{x} &= \dfrac{1}{3} \Big( - a_{1} x \varepsilon + \dfrac{\sqrt{-c^{2}+3d^{2}}y(-3c+a_{1}\varepsilon)}{c}\Big), \\ 
\dot{y} &= \dfrac{1}{3} \Bigg( - 2 a_{1} y \varepsilon + 6 c^{2} \Big( \dfrac{3 c w}{(- c^{2} + 3 d^{2})^{3/2}} + \dfrac{y}{c^{2}-3d^{2}} \Big) z \varepsilon + \dfrac{x \big( - 3 c^{2} + 9 d^{2} + 2 c (a_{1} - 3 z) \varepsilon \big) }{\sqrt{- c^{2} + 3 d^{2}}}  \Bigg), \\
\dot{z} &= b_{1} \Big( c+\dfrac{d^{2}}{c}-e \Big) + \dfrac{2 \varepsilon}{9 c (c^{2}-3d^{2})^{2}} \Big( - 3 d^{2} x + c^{2} (3 w + x) - c \sqrt{- c^{2} + 3 d^{2}} y \Big) \Big( - 3 c d^{2} x + 		\\ &\quad 	c^{3} (6 w + x) - (c^{2} + 3 d^{2}) \sqrt{- c^{2} + 3 d^{2}} y \Big) - b_{1} z \varepsilon, \\ 
\dot{w} &= \dfrac{1}{6} \Bigg( \dfrac{a_{1} (c^{2} + d^{2}) \Big( c x - \sqrt{- c^{2} + 3 d^{2}} y \Big)}{c^{3}} + \dfrac{4 \Big( 3 d^{2} x - c^{2}(3 w + x) + c \sqrt{- c^{2} + 3 d^{2}} y \Big) z}{c^{2} - 3 d^{2}} \Bigg) \varepsilon. 
 \end{aligned}
\end{equation}
Performing the cylindrical change of variables, 
\begin{equation}\label{p5}
 \begin{aligned}
(x,y,z,w) \mapsto (r \cos \theta , r \sin \theta , z , w)
 \end{aligned}
\end{equation}

the system (\ref{p4}) becomes 

\begin{equation}\label{p6} 
 \begin{aligned}  
\dfrac{d r}{d \theta} &= - \dfrac{\varepsilon}{3 c (c^{2}-3d^{2})^{2}} \Bigg( a_{1}c(-c^{2}+3d^{2})^{3/2} r \cos^{2} + (c^{2}-3d^{2}) r \Big( a_{1} (c^{2}+3d^{2})-6c^{2}z \Big)   	\\
&\quad  	 \cos \theta \sin \theta + 2 c \sin \theta \Big( - 9 c^{3} w z + \sqrt{-c^{2}+3d^{2}} r (- a_{1} c^{2} + 3 a_{1} d^{2} + 3 c^{2} z) \sin \theta \Big) \Bigg) + O(\varepsilon^{2})  \\
	&= \varepsilon F_{1} (\theta, r, z, w) + O(\varepsilon^{2}),  \\
\dfrac{d z}{d \theta} &= \dfrac{\varepsilon}{9 \sqrt{-c^{2}+3d^{2}} (c^{3}-3cd^{2})^{2} r} \Bigg( 3 c^{2} r \Big( - 2 d^{2} (c^{2}-3d^{2}) r^{2} + 12 c^{4} w^{2} - 3 b_{1} (c^{2}-3d^{2})^{2} z \Big)  	 	\\
&\quad  	 + 18 c^{4} w \Big( (c^{2}-3d^{2}) r^{2} - 3 b_{1} (c^{2}+d^{2}-ce) z \Big) \cos \theta + 6 b_{1} c^{2} (c^{2}-3d^{2}) (c^{2}+d^{2}-ce) r  	\\
&\quad  	    (a_{1}-3z) \cos^{2} \theta + 2 c^{4} (c^{2}-3d^{2}) r^{3} \cos 2 \theta + c \sqrt{-c^{2}+3d^{2}} r \Big( - (c^{2}-3d^{2}) \big( 3 a_{1} b_{1} (c^{2} +     	\\
&\quad  	 d^{2}-ce) + 2 (2c^{2}+3d^{2}) r^{2} \big) + 18 b_{1} c^{2} (c^{2}+d^{2}-ce) z \Big) \cos \theta \sin \theta + 3 r \sin \theta \Big( - 6 c^{3} (c^{2}	 \\
&\quad  	  + d^{2}) \sqrt{-c^{2}+3d^{2}} r w + a_{1} b_{1} (c^{2}-3d^{2})^{2}    	(c^{2}+d^{2}-ce) \sin \theta \Big) \Bigg) + O(\varepsilon^{2})\\
&= \varepsilon F_{2} (\theta, r, z, w) + O(\varepsilon^{2}),  \\	
\dfrac{d w}{d \theta} &= - \dfrac{\varepsilon}{6 c^{3} (- c^{2}+3d^{2})^{3/2}} \Bigg( \sqrt{-c^{2}+3d^{2}} r \Big( - a_{1} (c^{2}-3d^{2}) (c^{2}+d^{2}) + 4 c^{4} z \Big) \sin \theta      	\\
&\quad  	  - 12 c^{5} w z + c (c^{2}-3d^{2}) r \Big(  a_{1} (c^{2}+d^{2}) - 4 c^{2} z \Big) \cos \theta	 \Bigg) + O(\varepsilon^{2})\\		
	&= \varepsilon F_{3} (\theta, r, z, w) + O(\varepsilon^{2}). 
 \end{aligned}
\end{equation}
System (\ref{p6}) is written in the normal form (\ref{aver-1}) for applying the averaging theory and satisfies all the assumptions of Theorem \ref{teor3} . Then, using the notations of the averaging theory described in Theorem \ref{teor3}, we have $t=\theta$, $T=2\pi$, $x=(r,z,w)$,
$$
F(\theta, r, z, w) = \left( \begin{array}{c}
F_{1}(\theta, r, z, w) \\
F_{2}(\theta, r, z, w) \\
F_{3}(\theta, r, z, w)
\end{array}\right)
\mbox{,} \hspace{0.2cm} and \hspace{0.2cm}
f(r, z, w) = \left( \begin{array}{c}
f_{1}(r, z, w) \\
f_{2}(r, z, w) \\
f_{3}(r, z, w)
\end{array}\right)
$$
Then we compute the integrals, i.e.
\begin{equation*}
 \begin{aligned}
f_{1}(r, z, w) &= \frac{1}{2 \pi} \int^{2 \pi}_{0} F_{1}(\theta, r, z, w) d \theta = \frac{r \Big( a_{1} (c^{2}-3d^{2}) - 2 c^{2} z \Big)}{2 (-c^{2}+3d^{2})^{3/2}} , \\ 
f_{2}(r, z, w) &= \frac{1}{2 \pi} \int^{2 \pi}_{0} F_{2}(\theta, r, z, w) d \theta \\
			  &= \frac{1}{6 \sqrt{-c^{2}+3d^{2}} (c^{3}-3cd^{2})^{2}} \Bigg(  3 a_{1} b_{1} (c^{4}-4c^{2}d^{2}+3d^{2})(c^{2}+d^{2}+3d^{4}) 2c^{2} 	\\ &\quad 	 + \Big( - 2 d^{2} (c^{2}-3d^{2}) r^{2} + 12 c^{4} w^{2} - 3 b_{1} (c^{2}-3d^{2}) (2 c^{2}-2 d^{2} - c e) z \Big) \Bigg) , \\ 
f_{3}(r, z, w) &= \frac{1}{2 \pi} \int^{2 \pi}_{0} F_{3}(\theta, r, z, w) d \theta = \frac{2 c^{2} w z}{(-c^{2}+3d^{2})^{3/2}} .
 \end{aligned} 
\end{equation*}
Solving the equations $f_{1}(r, z, w) = f_{2}(r, z, w) = f_{3}(r, z, w) = 0$, we can get the following five solutions:
\begin{equation*}
 \begin{aligned}
s_{1} &= \Bigg( 0, \hspace{0.2cm} \dfrac{a_{1}(c^{2}-d^{2})(c^{2}+d^{2}-ce)}{2c^{2}(2c^{2}-2d^{2}-ce)}, \hspace{0.2cm} 0 \Bigg), \\
s_{2,3} &= \Bigg( \pm \dfrac{\sqrt{3} \sqrt{a_{1}} \sqrt{- b_{1} \Big( c^{4} - 8 c^{2} d^{2} + 7 d^{4} + 2 c d^{2} e \Big)}}{2cd} , \hspace{0.2cm} \dfrac{1}{2} a_{1} \Big( 1 - \dfrac{3d^{2}}{c^{2}} \Big) , \hspace{0.2cm} 0 \Bigg), \\
s_{4,5} &= \Bigg( 0, \hspace{0.2cm} 0 , \hspace{0.2cm} \pm \dfrac{\sqrt{a_{1}} \sqrt{-b_{1} \Big( c^{4} - 4 c^{2} d^{2} + 3 d^{4} \Big) \Big( c^{2} + d^{2} - c e \Big)}}{2 \sqrt{2} c^{3}} \Bigg).
  \end{aligned}
\end{equation*}
The solution $ s_{j}$, $j=1,\ldots, 5$ exist  if only if $c \neq 0$, $d \neq 0$, and $2 (c^2-d^2)-ce \neq 0$. On the other hand, the solution  $s_{1} \neq (0,0,0)$ if only if $(c^2-d^2)(c^2+d^2-c e) \neq 0$, and the solutions $s_{2,3}$  and $s_{4,5}$ are real if only if $c^{4} - 8 c^{2} d^{2} + 7 d^{4} + 2 c d^{2} e <0$ and $ (c^{4} - 4 c^{2} d^{2} + 3 d^{4})(c^{2} + d^{2} - c e )<0$.

For the five solutions, we get

\begin{equation*}
 \begin{aligned}
\det \Big( \dfrac{\partial f}{\partial \mathbf{x}} (s_{1}) \Big) &=  \dfrac{a_{1}^{2} b_{1} (c^{2}-d^{2}) (c^{2}+d^{2}-ce) (c^{4}-8 c^{2} d^{2} + 7 d^{4} + 2 c d^{2} e)}{2 (-c^{2}+3d^{2})^{9/2} (2c^{2}-2d^{2}-ce)}, \\ 
\det \Big( \dfrac{\partial f}{\partial \mathbf{x}} (s_{2}) \Big) &= \det \Big( \dfrac{\partial f}{\partial \mathbf{x}} (s_{3}) \Big) \\
&= \dfrac{a_{1}^{2} b_{1} \big( c^{4} - 8 c^{2} d^{2} + 7 d^{2} + 2 c d^{2} e \big)}{(-c^{2}+3d^{2})^{7/2}}, \\
\det \Big( \dfrac{\partial f}{\partial \mathbf{x}} (s_{4}) \Big) &= \det \Big( \dfrac{\partial f}{\partial \mathbf{x}} (s_{5}) \Big) \\
&= \dfrac{a_{1}^{2} b_{1} (c^{2}-d^{2}) \big( c^{2} + d^{2} - 2 c e \big)}{(-c^{2}+3d^{2})^{7/2}}.
  \end{aligned}
\end{equation*}

When $a_1 \neq 0$, $b_1 \neq 0$  and $3 d^2-c^2>0$ then $ \det \Big( \dfrac{\partial f}{\partial \mathbf{x}} (s_{j}) \Big) \neq 0 $, for each $j=1, \ldots, 5$.  Then according to Theorem \ref{teor3}, we see that the system (\ref{p6}) has one periodic solution $\mathbf{x}_{j}(\theta , \varepsilon)$ such that $\mathbf{x}_{j}(0 , \varepsilon) = s_{j} + O(\varepsilon)$, for each $j=1, \ldots, 5$. Bring the solution back to the system (\ref{p4}), and we have one periodic solution 
$\Phi_{j}(t,\varepsilon) = \big( X_{j}(t,\varepsilon), Y_{j}(t,\varepsilon), Z_{j}(t,\varepsilon), W_{j}(t,\varepsilon) \big)$.
Then the system (\ref{p2}) has the periodic solution $\varepsilon \Phi_{j}(t,\varepsilon)$, $j=1,\ldots,5$.

To determine the stability of the periodic solution one needs to calculate eigenvalues of the Jacobian matrix $\partial F (s_{j})/\partial x$,$\quad j=1, \ldots, 5$.

\medskip

The Jacobian matrices $\partial F (s_{1})/\partial x$ have the same characteristic equation, 
\begin{equation*}
 \begin{aligned}
\lambda^{3} + \Theta_{1} \lambda^{2} - \Theta_{2} \lambda - \Theta_{3}
  \end{aligned}
\end{equation*}
where $\Theta_{1}$, $\Theta_{2}$ and $\Theta_{3}$ are 
\begin{equation*}
 \begin{aligned}
\Theta_{1} &= \dfrac{a_{1} \Big( - 3 c^{4} + 8 c^{2} d^{2} - 5 d^{4} + 2 c^{3} e - 4 c d^{2} e \Big) - 2 b_{1} \Big( 2 d^{2} + c (- 2 c + e)\Big)^{2}}{2 (-c^{2}+3d^{2})^{3/2} (2 c^{2}-2 d^{2} - c e)}, \\
\Theta_{2} &= \dfrac{a_{1} }{{2 (c^{2}-3d^{2})^{3} \Big(2 d^{2} + c (- 2 c + e) \Big)^{2}}} \Big( a_{1} (c^{2}-d^{2}) (c^{2}+d^{2}-ce) (c^{4}-8c^{2}d^{2}+7d^{4} \\ 
& + 2cd^{2}e) + b_{1} (3 c^{4}-8c^{2}d^{2}+5d^{4}-2c^{3}e+4cd^{2}e) \big( 2d^{2}+c(-2c+e) \big)^{2} \Big), \\
\Theta_{3} &= \dfrac{a_{1}^{2} b_{1} (c^{2}-d^{2}) (c^{2}+d^{2}-ce) (c^{4}-8c^{2}d^{2}+7d^{4}+2cd^{2}e)}{2 (-c^{2}+3d^{2})^{9/2} (2 c^{2}-2 d^{2} - c e)}.
  \end{aligned}
\end{equation*}
The eigenvalues are  given as follows:
\begin{equation*}
 \begin{aligned}
\lambda_{1} &= - \dfrac{b_{1} \Big( 2d^{2}+c(-2c+e) \Big)}{(-c^{2}+3d^{2})^{3/2}}, \hspace{0,2cm} \lambda_{2} = \dfrac{a_{1} (c^{4}-8c^{2}d^{2}+7d^{4}+2cd^{2}e)}{2(-c^{2}+3d^{2})^{3/2} (2c^{2}-2d^{2}-ce)}, \hspace{0,2cm} \\ 
\lambda_{3} &= \dfrac{a_{1} \Big( - c^{4}+d^{4}+c^{3}e-c d^{2}e \Big)}{(-c^{2}+3d^{2})^{3/2} \Big( 2d^{2}+c(-2c+e) \Big)}.
  \end{aligned}
\end{equation*}

\medskip

The Jacobian matrices $\partial F (s_{2})/\partial x$ and $\partial F (s_{3})/\partial x$ have the same characteristic equation, 
\begin{equation*}
 \begin{aligned}
\lambda^{3} - \Upsilon_{1} \lambda^{2} - \Upsilon_{2} \lambda - \Upsilon_{3}
  \end{aligned}
\end{equation*}
where $\Upsilon_{1}$, $\Upsilon_{2}$ and $\Upsilon_{3}$ are 
\begin{equation*}
 \begin{aligned}
\Upsilon_{1} &= \dfrac{a_{1}(c^{2}-3d^{2}) + b_{1}(2c^{2}-2d^{2}-ce)}{(-c^{2}+3d^{2})^{3/2}} \\
\Upsilon_{2} &= \dfrac{a_{1} b_{1} (c^{2}-d^{2}) (c^{2}+d^{2}-ce)}{(c^{2}-3d^{2})^{3}}\\
\Upsilon_{3} &= \dfrac{a_{1}^{2} b_{1} (c^{4}-8c^{2}d^{2}+7d^{4}+2cd^{2}e)}{(-c^{2}+3d^{2})^{7/2}}
  \end{aligned}
\end{equation*}
The eigenvalues are  given as follows:
\begin{equation*}
 \begin{aligned}
\widehat{\lambda_{1}} &= - \dfrac{a_{1}}{\sqrt{-c^{2}+3d^{2}}}, \\ 
\widehat{\lambda_{2,3}} &=  \resizebox{0.9\hsize}{!}{$ - \dfrac{2b_{1}d^{2}+b_{1}c(-2c+e) \pm i \sqrt{b_{1} \Big( -4a_{1} (c^{4}-8c^{2}d^{2}+7d^{4}+2cd^{2}e) - b_{1} \big( 2d^{2}+c(-2c+e) \big)^{2} \Big)}}{2(-c^{2}+3d^{2})^{3/2}} $}.
  \end{aligned}
\end{equation*}

The Jacobian matrices $\partial F (s_{4})/\partial x$ and $\partial F (s_{5})/\partial x$ have the same characteristic equation, 
\begin{equation*}
 \begin{aligned}
\lambda^{3} - \Gamma_{1} \lambda^{2} + \Gamma_{2} \lambda - \Gamma_{3} 
  \end{aligned}
\end{equation*}
where $\Gamma_{1}$, $\Gamma_{2}$ and $\Gamma_{3}$ are 
\begin{equation*}
 \begin{aligned}
\Gamma_{1} &= \dfrac{a_{1} c^{2} + 4 b_{1} c^{2} - 3 a_{1} d^{2} - 4 b_{1} d^{2} - 2 b_{1} c e}{2(-c^{2}+3d^{2})^{3/2}}, \\
\Gamma_{2} &= \dfrac{a_{1} b_{1} (2 c^{4} + 8c^{2}d^{2} - 10d^{4} - 3 c^{3} e + c d^{2} e)}{2(c^{2}-3d^{2})^{3}}, \\
\Gamma_{3} &= \dfrac{a_{1}^{2} b_{1} (c^{2}-d^{2}) (c^{2}+d^{2}-ce)}{(-c^{2}+3d^{2})^{7/2}}.
  \end{aligned}
\end{equation*}
The eigenvalues are  given as follows:
\begin{equation*}
 \begin{aligned}
\widetilde{\lambda_{1}} &= - \dfrac{a_{1}}{2 \sqrt{-c^{2}+3d^{2}}}, \\ 
\widetilde{\lambda_{4,5}} &= \resizebox{0.9\hsize}{!}{$ - \dfrac{2b_{1}d^{2}+b_{1}c(-2c+e) \pm  i \sqrt{b_{1} \Big( 8 a_{1} (- c^{4}+d^{4}+c^{3}e-cd^{2}e) - b_{1} \big( 2d^{2}+c(-2c+e) \big)^{2} \Big)}}{2(-c^{2}+3d^{2})^{3/2}} $}.
  \end{aligned}
\end{equation*}

We have that $\lambda_{1},\widehat{\lambda_{1}},\widetilde{\lambda_{1}} $ is real and $\lambda_{2,3},\widehat{\lambda_{2,3}} ,\widetilde{\lambda_{4,5}}$ are complex numbers if $16 \eta +3 b_1 \omega^2<0$  and $4 \eta_{1}+3 b_1 \omega^2<0$. In this case, since that $a_1>0, b_1>0$, $(c^{4}-8c^{2}d^{2}+7d^{4}+2cd^{2}e)<0$, $2c^{2}-2d^{2}-ce<0$ and $c^4-d^4-c^3 e+ cd^2 e >0$ , then  this implies that the periodic orbits $\varepsilon \Phi_{j}(t,\varepsilon) $, $j\in \lbrace 1, \ldots 5 \rbrace$ are stable. 
\end{proof}


\begin{proof}[ Proof. of statement (iii) of Theorem \ref{teor2}]


Let 

$$(a,b,d) = (- 2 c + \varepsilon a_{1} , \varepsilon b_{1}, - \dfrac{\sqrt{c^{2}+\omega^{2}}}{\sqrt{3}} + \varepsilon d_{1}),$$

where $\omega > 0$ and $\varepsilon > 0$ are sufficiently small parameter. Them, we translate $\mathtt{p_{\pm}}$ to the origin the coordinates doing system (\ref{s1}) becomes $(x,y,z,w)=(\overline{x},\overline{y},\overline{z},\overline{w})+ \mathtt{p_{\pm}}$, then we introduce the scaling of variables $(\overline{x},\overline{y},\overline{z},\overline{w}) = (\varepsilon x,\varepsilon y,\varepsilon z,\varepsilon w)$, with these changes of variables  system (\ref{s1}) can be written as 
%
%

\begin{equation} \label{pp2}
 \begin{aligned} 
\dot{x} &= (x-y) (2c-a_{1} \varepsilon), \\ 
\dot{y} &= \dfrac{1}{3} \Bigg( c(4x-3y)-3wx \varepsilon - z \Big( 3 d_{1} \varepsilon + \sqrt{3} \sqrt{c^{2}+\omega^{2}} \Big) + \dfrac{b_{1} w \sqrt{\varepsilon} }{\sqrt{c} \sqrt{-b_{1}(4c^{2}-3ce+\omega^{2})}} \Big( 3 d_{1} \\
&\quad \varepsilon \sqrt{c^{2}+\omega^{2}} + \sqrt{3} (4 c^{2} - 3 c e + \omega^{2}) \Big)+ \dfrac{x \Big( \omega^{2} + d_{1} \varepsilon \big( 3 d_{1} \varepsilon + 2 \sqrt{3} \sqrt{c^{2}+\omega^{2}}\big) \Big)}{c} \Bigg), \\ 
\dot{z} &= - c z + d_{1} y \varepsilon - \dfrac{y \sqrt{c^{2}+\omega^{2}}}{\sqrt{3}},  \\
\dot{w} &= (-b_{1}w+xy) \varepsilon - \dfrac{b_{1}d_{1}(x+y)\varepsilon^{3/2}\sqrt{c^{2+\omega^{2}}}}{\sqrt{c}\sqrt{-b_{1}(4c^{2}-3ce+\omega^{2})}} + \dfrac{(x+y)\sqrt{\varepsilon}\sqrt{-b_{1}(4c^{2}-3ce+\omega^{2})}}{\sqrt{3}\sqrt{c}}. 
  \end{aligned}
\end{equation}

After the linear change in variables $(x,y,z,w) \mapsto (X, Y, Z, W)$, 

\begin{equation}\label{pp3}
 \begin{aligned}
x &= \dfrac{2 c \Big( - \sqrt{3} c Z + c Y \omega + X \omega^{2} \Big)}{\sqrt{3} \omega^{2}}, \\ 
y &= \dfrac{c X \omega^{2} + Y \omega^{3} + c^{2} \Big(- 2 \sqrt{3} Z + 2 Y \omega \Big)}{\sqrt{3} \omega^{2}}, \\
z &= \dfrac{\sqrt{c^{2}+\omega^{2}} \Big( 2 \sqrt{3} c Z - 2 c Y \omega + X \omega^{2} \Big)}{3 \omega^{2}}, \\ 
w &= W . 
  \end{aligned}
\end{equation}

the linear part at the origin of system (\ref{pp2}) for $\varepsilon = 0$ can be transformed into its real Jordan normal form,

$$\left( \begin{array}{cccc}
0 & - \omega & 0 & 0 \\
\omega & 0 & 0 & 0 \\
0 & 0 & 0 & 0 \\
0 & 0 & 0 & 0
\end{array}\right).$$

Under the change in variable (\ref{pp3}), where we have written $(x,y,z,w)$ instead of $(X,Y,Z,W)$ the system (\ref{pp2}) can be written as
\begin{equation}\label{pp4}
 \begin{aligned}
\dot{x} &= - y \omega + \dfrac{1}{3} \varepsilon \Bigg( a_{1} \Big( - x + \dfrac{y \omega}{c} \Big) + \dfrac{- 6 c^{2} d_{1} z + \sqrt{3} d_{1} \omega ( 2 c^{2} y + c x \omega + y \omega^{2})}{\omega^{2} \sqrt{c^{2}+\omega^{2}}} \Bigg), \\ 
\dot{y} &= \dfrac{1}{3 \omega^{3}} \Bigg( 3(-2cw+\omega^{2}) x \omega^{2} + 6 c^{2} w (\sqrt{3} z - y \omega) + 6 d_{1}^{2} \varepsilon^{2} (- \sqrt{3} c z + c y \omega + x \omega^{2}) - \dfrac{3 w \omega^{2}}{\sqrt{c}} 	 	\\&\quad  	  \sqrt{- b_{1} \varepsilon \Big( 4 c^{2} - 3 c e + \omega^{2} + d_{1} \varepsilon \big( 3 d_{1} \varepsilon - 2 \sqrt{3} \sqrt{c^{2}+\omega^{2}} \big) \Big)}  - \dfrac{\varepsilon \omega^{2}}{\sqrt{c^{2}+\omega^{2}}} \Big( 4 \sqrt{3} c^{2} d_{1} x + c	 	\\&\quad  	   \big( - 6 d_{1} z + \sqrt{3} d_{1} y \omega - 2 a_{1} x \sqrt{c^{2}+\omega^{2}} \big) + \omega \big( 5 \sqrt{3} d_{1} x \omega + 2 a_{1} y \sqrt{c^{2}+\omega^{2}} \big) \Big) \Bigg), \\
\dot{z} &= - \dfrac{2 c w x}{\sqrt{3}} + \dfrac{2 c^{2} w z}{\omega^{2}} - \dfrac{2 c^{2} w y}{\sqrt{3} \omega} + \dfrac{2 d_{1}^{2} \varepsilon^{2} \big( - 3 c z + \sqrt{3} \omega (c y + x \omega) \big)}{3 \omega^{2}} + \dfrac{\varepsilon}{18 c^{2}} \Big( 6 c d_{1} (- 4 c x + y \omega) 	 	\\&\quad  	   \sqrt{c^{2}+\omega^{2}} + \sqrt{3} a_{1} (c x - y \omega) (4 c^{2} + \omega^{2}) \Big) - \dfrac{w}{\sqrt{c}} \sqrt{- b_{1} \varepsilon \Big( c^{2} -  c e + \big( d_{1} \varepsilon - \dfrac{\sqrt{c^{2}+\omega^{2}}}{\sqrt{3}} \big)^{2} \Big)}, \\ 
\dot{w} &= - b_{1} w \varepsilon + \dfrac{1}{3 \sqrt{c}} \Bigg( 2 c^{5/2} \Big( x^{2} + y \big( y - \dfrac{\sqrt{3} z}{\omega} \big) \Big) + 2 c^{3/2} x y \omega + \dfrac{6 c^{7/2} x (- \sqrt{3} z + y \omega)}{\omega^{2}} + \dfrac{4 c^{9/2}}{\omega^{4}}   	\\&\quad  	  \big( 3 z^{2} - 2 \sqrt{3} y z \omega + y^{2} \omega^{2} \big) + 3 c x \sqrt{- b_{1} \varepsilon \Big( 4 c^{2} - 3 c e + \omega^{2} + d_{1} \varepsilon \big( 3 d_{1} \varepsilon - 2 \sqrt{3} \sqrt{c^{2}+\omega^{2}} \big) \Big)} +  	\\&\quad  	   y \omega \sqrt{ - b_{1} \varepsilon \Big( 4 c^{2} - 3 c e + \omega^{2} + d_{1} \varepsilon \big( 3 d_{1} \varepsilon - 2 \sqrt{3} \sqrt{c^{2}+\omega^{2}} \big) \Big)} + \dfrac{4 c^{2} (- \sqrt{3} z + y \omega)}{\omega^{2}}  	\\&\quad  	\sqrt{- b_{1} \varepsilon \Big( 4 c^{2} - 3 c e + \omega^{2} + d_{1} \varepsilon \big( 3 d_{1} \varepsilon - 2 \sqrt{3} \sqrt{c^{2}+\omega^{2}} \big) \Big)}   \Bigg) . 
  \end{aligned}
\end{equation}

Performing the cylindrical change of variables, 
\begin{eqnarray*}
(x,y,z,w) \mapsto (r \cos \theta , r \sin \theta , z , w)
\end{eqnarray*}
system (\ref{pp4}) becomes 
\begin{equation}\label{pp6} 
	\begin{aligned}  
\dfrac{d r}{d \theta} &= \dfrac{1}{3cr \omega^{4} \sqrt{c^{2}+\omega^{2}}} \Bigg( - a_{1} c r^{2} \omega^{3} \sqrt{c^{2}+\omega^{2}} \cos^{2} \theta + \cos \theta \bigg( - 6 c^{3} d_{1} r z \omega 	\\ 
&\quad	+ \Big( - 2 \sqrt{3} c^{3} d_{1} r^{2} \omega^{2} +  \omega^{2} \sqrt{c^{2}+\omega^{2}} \big( 3 b_{1} w^{2} + a_{1} r^{2} \omega^{2} \big) + 2 c^{2}  \sqrt{c^{2}+\omega^{2}} \big( 6 b_{1} w^{2}\\ 
&\quad  + r^{2} (a_{1}-3w) \omega^{2} \big) - c \big( 4 \sqrt{3} d_{1} r^{2} \omega^{4}+	9 b_{1} e w^{2} \sqrt{c^{2}+\omega^{2}} \big) \Big) \sin \theta \bigg) + c r \bigg( \sqrt{3} c d_{1} \\ 
&\quad r \omega^{3} \cos 2 \theta + 6 c z \Big( d_{1} \omega^{2} + \sqrt{3} c w \sqrt{c^{2}+\omega^{2}} \Big) \sin \theta - 2 r \omega \sqrt{c^{2}+\omega^{2}} (3 c^{2} w \\
&\quad + a_{1} \omega^{2}) \sin^{2} \theta \bigg) \Bigg) , \\ 
\dfrac{d z}{d \theta} &= \dfrac{1}{18 c^{2} r \omega^{3}} \Bigg( c \bigg( 6 \sqrt{3} b_{1} w^{2} (4 c^{2}-3ce+\omega^{2}) + r^{2} \omega^{2} \Big( - 24 c d_{1} \sqrt{c^{2}+\omega^{2}} \\
 &\quad + \sqrt{3} \big( 4 c^{2} (a_{1}-3w) +  a_{1} \omega^{2} \big) \Big) \bigg) \cos \theta + r \bigg( 36 c^{4} w z + r \omega \Big( 6 c d_{1} \omega^{2} \sqrt{c^{2}+\omega^{2}} \\
 &\quad - \sqrt{3} \big( 12 c^{4} w + 4 a_{1} c^{2} \omega^{2} + a_{1} \omega^{4} \big) \Big) \sin \theta \bigg) \Bigg) , \\ 
\dfrac{d w}{d \theta} &= \dfrac{1}{3 c r \omega^{5}} \Bigg( 12 c^{5} r z^{2} + c r (2 c^{2} r^{2} - 3 b_{1} w) \omega^{4} + c^{2} r^{2} \omega \Big( 6 c^{2} \omega \cos \theta (- \sqrt{3} z + r \omega \sin \theta)  \\
 &\quad	 	c z (4 c^{2}+\omega^{2}) + r \omega^{4} \cos \theta + 2 c^{3} r \omega \sin \theta \big) \Big) - b_{1} w (4c^{2}-3ce+\omega^{2}) \cos \theta \Big( - 4 \sqrt{3} c^{2} z 	 	\\
  &\quad	+ r \omega \big( 3 c \omega \cos \theta    + (4 c^{2} + \omega^{2}) \sin \theta \big) \Big) \Bigg)  . 
	\end{aligned}
\end{equation}
System (\ref{pp6}) is written in the normal form (\ref{aver-1}) for applying the averaging theory and satisfies all the assumptions of Theorem \ref{teor3}. Then, using the notations of the averaging theory described in Theorem \ref{teor3}, we have $t=\theta$, $T=2\pi$, $x=(r,z,w)$,
$$
F(\theta, r, z, w) = \left( \begin{array}{c}
F_{1}(\theta, r, z, w) \\
F_{2}(\theta, r, z, w) \\
F_{3}(\theta, r, z, w)
\end{array}\right)
\mbox{,} \hspace{0.2cm} and \hspace{0.2cm}
f(r, z, w) = \left( \begin{array}{c}
f_{1}(r, z, w) \\
f_{2}(r, z, w) \\
f_{3}(r, z, w)
\end{array}\right).
$$
Then we compute the integrals, i.e.
\begin{equation*}
 \begin{aligned}
f_{1}(r, z, w) &= - \dfrac{r (2 c^{2} w + a_{1} \omega^{2})}{2 \omega^{3}} , \\ 
f_{2}(r, z, w) &= \dfrac{2 c^{2} w z}{\omega^{3}} , \\ 
f_{3}(r, z, w) &= \dfrac{24 c^{4} z^{2} + c (4 c^{3} r^{2} - 12 b_{1} c w + 9 b_{1} e w) \omega^{2} + (4 c^{2} r^{2} - 9 b_{1} w) \omega^{4}}{6 \omega^{5}} .
 \end{aligned} 
\end{equation*}

Solving the equations $f_{1}(r, z, w) = f_{2}(r, z, w) = f_{3}(r, z, w) = 0$, we can get the following three solutions:

\begin{equation*}
 \begin{aligned}
s_{0} &= \big( 0, \hspace{0.2cm} 0, \hspace{0.2cm} 0 \big), \\
s_{1,2} &= \Bigg( \pm \dfrac{1}{2c^{2}} \sqrt{\dfrac{3}{2}} \sqrt{a_{1}} \omega \sqrt{\dfrac{b_{1} (- 4 c^{2} + 3 c e - 3 \omega^{2})}{c^{2} + \omega^{2}}} , \hspace{0.2cm} 0 , \hspace{0.2cm} - \dfrac{a_{1} \omega^{2}}{2 c^{2}} \Bigg), \\
  \end{aligned}
\end{equation*}

For two solutions, we get

\begin{equation*}
 \begin{aligned}
\det \Big( \dfrac{\partial f}{\partial \mathbf{x}} (s_{1}) \Big) &= \det \Big( \dfrac{\partial f}{\partial \mathbf{x}} (s_{2}) \Big) \\
&= \dfrac{a_{1}^{2} b_{1} \big( 4 c^{2} - 3 c e + 3 \omega^{2} \big)}{2 \omega^{5}}.
  \end{aligned}
\end{equation*}

When $c\neq 0$, $a_{1} \neq 0$, and $\kappa=b_{1}(4 c^2-3 ce+3 \omega^2)< 0$ then $\det \Big( \frac{\partial f}{\partial \mathbf{x}} (s_{j}) \Big) \neq 0$, $j=1,2$. 
Then according to Theorem \ref{teor3}, we see that the system (\ref{pp6}) has one periodic solution $\mathbf{x}_{j}(\theta , \varepsilon)$ such that $\mathbf{x}_{j}(0 , \varepsilon) = s_{j} + O(\varepsilon)$, $j=1,2$. Bring the solution back to the system (\ref{pp4}), and we have one periodic solution
 $\Phi_{j}(t,\varepsilon) = \big( X_{j}(t,\varepsilon), Y_{j}(t,\varepsilon), Z_{j}(t,\varepsilon), W_{j}(t,\varepsilon) \big)$. Then the system (\ref{pp2}) has the periodic solution $\varepsilon \Phi_{j}(t,\varepsilon)$, $j=1,2$. 

\medskip

The Jacobian matrices $\partial F (s_{1})/\partial x$ have the same characteristic equation, 
\begin{equation*}
 \begin{aligned}
\lambda^{3} + \dfrac{b_{1} c (4 c - 3 e) + (2 a_{1} + 3 b_{1}) \omega^{2}}{2 \omega^{3}} \lambda^{2} - \dfrac{a_{1}^{2} b_{1} (4 c^{2} - 3 c e + 3 \omega^{2})}{2 \omega^{5}}.
  \end{aligned}
\end{equation*}

The eigenvalues are  given as follows:
\begin{equation*}
 \begin{aligned}
\lambda_{1} &= - \dfrac{a_{1}}{\omega},  \\
\lambda_{2} &=  - \dfrac{1}{4 \omega^{3}} \Bigg( b_{1} (4 c^{2}-3ce+3\omega^{2}) +  \sqrt{b_{1} (4 c^{2} - 3 c e + 3 \omega^{2}) \Big( b_{1} c (4 c - 3 e) + (8 a_{1} + 3 b_{1}) \omega^{2} \Big)} \Bigg) , \hspace{0,2cm} \\ 
\lambda_{3} &=  - \dfrac{1}{4 \omega^{3}} \Bigg( b_{1} (4 c^{2}-3ce+3\omega^{2}) -  \sqrt{b_{1} (4 c^{2} - 3 c e + 3 \omega^{2}) \Big( b_{1} c (4 c - 3 e) + (8 a_{1} + 3 b_{1}) \omega^{2} \Big)} \Bigg).
  \end{aligned}
\end{equation*} 

We have that $\lambda_1$  and $\lambda_{2,3}=- \dfrac{1}{4 \omega^{3}} (\kappa \pm \sqrt{\kappa(\kappa+8 a_1 \omega^2)})$ are reals, if $a_1>0$, $\kappa>0$ and regardless of the sign assumed by $\kappa(\kappa+8 a_1 \omega^2)$, at least one of the eigenvalues has a positive real part. In this case, the periodic solution $\varepsilon \Phi_{j}(t,\varepsilon)$, $j=1,2$ is unstable . 
\end{proof}


%
%

\smallskip


\section*{Acknowledgments}

The first author was financed in part by the Coordenação de Aperfeiçoamento de Pessoal de Nível Superior - Brasil (CAPES) - Finance Code 001.


\bibliographystyle{acm}
\bibliography{b.bib}

\begin{thebibliography}{1}

\bibitem{haken1975analogy}
{\sc Haken, H.}
\newblock Analogy between higher instabilities in fluids and lasers.
\newblock {\em Physics Letters A 53}, 1 (1975), 77--78.

\bibitem{natiq2019dynamics}
{\sc Natiq, H., Said, M. R.~M., Al-Saidi, N.~M., and Kilicman, A.}
\newblock Dynamics and complexity of a new 4d chaotic laser system.
\newblock {\em Entropy 21}, 1 (2019), 34.

\bibitem{sanders2007averaging}
{\sc Sanders, J.~A., Verhulst, F., and Murdock, J.}
\newblock {\em Averaging methods in nonlinear dynamical systems}, vol.~59.
\newblock Springer, 2007.

\bibitem{van1997nonlinear}
{\sc van Tartwijk, G.~H., and Agrawal, G.~P.}
\newblock Nonlinear dynamics in the generalized lorenz-haken model.
\newblock {\em Optics communications 133}, 1-6 (1997), 565--577.

\bibitem{verhulst2006nonlinear}
{\sc Verhulst, F.}
\newblock {\em Nonlinear differential equations and dynamical systems}.
\newblock Springer Science \& Business Media, 2006.

\end{thebibliography}


\end{document}